\documentclass[11pt]{article}
\usepackage{amsfonts,amsmath,tikz}
\usepackage{epsfig}
\usepackage{latexsym}
\usepackage{color}
\usepackage{epsf,amssymb}

\newtheorem{thm}{Theorem}[section]
\newtheorem{ob}[thm]{Observation}

\newtheorem{cor}[thm]{Corollary}
\newtheorem{lem}[thm]{Lemma}
\newtheorem{prop}[thm]{Proposition}

\def\vertex(#1){\put(#1){\circle*{2}}}
\def\vertexo(#1){\put(#1){\circle{2}}}
\def\vert(#1){\put(#1){\circle*{1.5}}}
\def\verto(#1){\put(#1){\circle{1.5}}}
\def\lab(#1)#2{\put(#1){\makebox(0,0)[c]{#2}}}
\newenvironment{proof}[1][Proof]{\textbf{#1.} }{\ \rule{0.5em}{0.5em}}

\newcommand{\cart}{\, \Box \,}

\newcommand{\rk}{\gamma_{{\rm r}k}}
\newcommand{\rik}{\gamma_{{\rm ri}k}}
\newcommand{\ri}{\gamma_{{\rm ri}2}}
\newcommand{\rd}{\gamma_{{\rm r}2}}

\newcommand{\cp}{\,\square\,}

\begin{document}

\title{On $k$-rainbow independent domination in graphs}

\author{
Tadeja Kraner \v{S}umenjak \footnote{ FKBV, University of Maribor,
Maribor, Slovenia. The author is also with the Institute  of
Mathematics, Physics and Mechanics, Jadranska 19, 1000 Ljubljana.
email: tadeja.kraner@um.si}  \and Douglas F. Rall \footnote{Furman University, Greenville, SC,
USA. e-mail: doug.rall@furman.edu}
 \and Aleksandra Tepeh\footnote{
FEECS, University of Maribor, Maribor, Slovenia. The author is also
with the Faculty of Information Studies, 8000 Novo Mesto, Slovenia.
email: aleksandra.tepeh@um.si }
 }

\date{}
\maketitle

\begin{abstract}

 In this paper, we define a new domination invariant on a graph $G$, which coincides with the ordinary independent domination
number of the generalized  prism $G \cart K_k$, called  the {\em
$k$-rainbow independent domination number} and denoted by $\rik(G)$. Some  bounds and exact values concerning this domination concept are determined. As a main result, we prove a  Nordhaus-Gaddum-type theorem on the sum for $2$-rainbow independent domination number, and show if G is a graph of order $n \ge 3$, then $5\leq \ri(G)+\ri(\overline{G})\leq n+3$, with both bounds being sharp.
\end{abstract}


\section{Introduction}

Domination in graphs has been an extensively researched branch of graph theory. Already in
\cite{fund-1998} more than 75 variations of domination were cited, and many more have been introduced since then.
This is not surprising as many of these concepts found applications in different fields,
for instance in  facility location problems, monitoring communication or
electrical networks, land surveying, computational biology, etc. Recent studies on domination and related concepts include also \cite{N2,du-2017,N1,  j-2017,shao}.
Although our original motivation for defining a new invariant arises from a desire to reduce the problem of computing the
independent domination number of the generalized prism $G \cart K_k$ to an integer labeling problem on a graph $G$, the new concept can also be seen as a model for a problem in combinatorial optimization.

In introducing this new concept we first mention that
in general we follow the notation and graph theory terminology
in~\cite{hik-2011}. Specifically, let $G$ be a finite, simple graph
with vertex set $V(G)$ and edge set $E(G)$. By $G\langle A\rangle$ we denote the subgraph of $G$
induced by the vertex set $A\subseteq V(G)$. For a positive integer $n$ with $n \ge2$ we denote an empty graph on $n$ vertices
by $N_n$ and a star of order $n$ by $S_n$.  The graph  obtained from $S_n$ by adding a single edge
is denoted $S_n^+$.
For any vertex $g$ in $G$, the {\em open neighborhood of
$g$}, written $N(g)$, is the set of vertices adjacent to $g$.  The
{\em closed neighborhood} of $g$ is the set $N[g] = N(g) \cup
\{g\}$. The vertex $g$ is a {\em universal} vertex of  $G$ if $N[g]=V(G)$.
If $A\subset V(G)$, then $N(A)$ (respectively, $N[A]$) denotes
the union of the open (closed) neighborhoods of all vertices of $A$. (In
the event that the graph $G$ under consideration is not clear we
write $N_G(g)$, and so on.)

The \emph{Cartesian product}, $G\Box H$, of graphs $G$ and $H$ is a
graph with $V(G\Box H)=V(G)\times V(H)$, where two vertices $(g,h)$ and
$(g',h')$ are adjacent in $G\Box H$ whenever
($gg'\in E(G)$ and $h=h'$) or ($g=g'$ and
$hh'\in E(H)$). For a fixed $h\in V(H)$ we call
$G^{h}=\{(g,h)\in V(G\Box H):g\in V(G)\}$ a $G$-\emph{layer} in
$G\Box H$. Similarly, an $H$-layer $^{g}\!H$ for a fixed $g\in V(G)$
is defined as $^{g}\!H=\{(g,h)\in V(G\Box H):h\in V(H)\}$. Notice
that the subgraph of $G\Box H$ induced by a $G$-layer or an
$H$-layer is isomorphic to $G$ or $H$, respectively.

If $A$ and $B$ are any two nonempty subsets of
$V(G)$ we say $A$ dominates $B$ if $B \subseteq N[A]$.  If $A$ dominates $V(G)$ then
 $A$ is a {\em dominating set} of $G$.  The {\em domination number} of $G$, denoted
by $\gamma(G)$, is the minimum cardinality of a dominating set of
$G$.  An {\em independent set} of a graph is a set of vertices, no two of
which are adjacent. An {\em independent dominating set} of $G$ is a
set that is both dominating and independent in $G$. This set is also called a {\em stable set}
or a {\em kernel} of the graph $G$. The {\em independent domination number}, $i(G)$, of a graph $G$ is the size of
a smallest independent dominating set. The {\em independence
number}  of $G$, denoted by $\alpha(G)$, is the maximum size of an
independent set in $G$. Observe that $\gamma(G) \leq i(G) \leq
\alpha(G)$.

\medskip
For a positive integer $k$ we denote the set $\{1,2,\ldots, k\}$ by
$[k]$.  In the remainder of this paper we will always assume
the vertex set of the complete graph $K_k$ is $[k]$.
The {\em power set} (that is, the set of all subsets) of
$[k]$ is denoted by $2^{[k]}$. Let $G$ be a graph and let $f$ be a
function that assigns to each vertex a subset of integers chosen
from the set $[k]$; that is, $f \colon V(G) \rightarrow 2^{[k]}$.
The {\em weight}, $\|f\|$, of  $f$ is defined as $\|f\| = \sum_{v\in
V(G)} |f(v)|$. The function $f$ is called a {\em $k$-rainbow
dominating function} ($k$RDF for short) of $G$ if for each vertex $v\in V(G)$ such that
$f(v)=\emptyset$ it is the case that
$$
\bigcup_{u\in N(v)} f(u) = \{1,\ldots,k\}\,.
$$

Given a graph $G$, the minimum
weight of a $k$-rainbow dominating function is called the {\em
$k$-rainbow domination number of $G$}, which we denote by $\rk(G)$.  Motivation for introducing this
concept arose from the observation that for $k\geq 1$ and for every graph $G$,
$\rk(G)=\gamma(G \cart K_k)$. See~\cite{bhr-2008}.  In other words, the problem of finding the
domination number of the Cartesian product $G \cart K_k$ is equivalent to an optimization problem involving
a restricted ``labeling'' of $V(G)$ with subsets of $[k]$.  In what follows we impose additional conditions
on this labeling in order to represent an independent dominating set of $G \cart K_k$.  Since the dominating set is
independent and the $K_k$-layers are complete, vertices cannot be labeled with a subset of cardinality more than 1.
This allows for each vertex of $G$ to be labeled by a single integer and leads us to the following.

\medskip

For a function $f: V(G) \to \{0,1,2,\ldots k\}$ we denote by $V_i$
the set of vertices to which the value $i$ is assigned by $f$, i.e.
$V_i=\{x\in V(G)\,:\, f(x)=i\}$. A function $f: V(G) \to
\{0,1,\ldots, k\}$ is called an {\em  $k$-rainbow independent
dominating function} ({\em $k$RiDF} for short) of $G$ if
the following two conditions hold:

\begin{enumerate}
\item $V_i$ is independent for $1 \le i \le k$, and
\item for every $x\in V_0$ it follows that $N(x)\cap V_i\neq \emptyset$, for every $i\in [k]$.
\end{enumerate}

Note that a $k$-rainbow independent dominating function $f$ can be
represented by the ordered partition $(V_0,V_1,\ldots,V_k)$ determined by $f$.
Hence, when it is convenient we simply work with the partition. The
weight of a $k$RiDF  $f$ is defined as $w(f)= \sum
_{i=1}^{k}|V_i|$, or equivalently $w(f)=n-|V_0|$, where $n$ is
the order of the graph. The {\em  $k$-rainbow independent domination
number} of a graph, denoted by $\rik(G)$, is the minimum
weight of a $k$RiDF of $G$. Note that
$\gamma_{ri1}(G)=i(G)$. A {\em $\rik(G)$-function} is a
$k$RiDF  of $G$ with weight $\rik(G)$, and a
{\em $k$RiDF-partition} of  $G$ is an ordered partition $(V_0,V_1,\ldots,V_k)$ of
$V(G)$ that represents a $\rik(G)$-function.

Chellali and Jafari Rad~\cite{cr-2015} introduced a graphical invariant using the name, independent
$2$-rainbow domination number, and later Shao et al.~\cite{shao} independently presented a natural generalization to the
invariant of Chellali and Jafari Rad for $k\ge 3$.  Our new invariant is quite different from these.
To explain these differences we provide their definition for the case $k=2$ and then to illustrate the differences we present
examples for the natural generalization for $k=3$.  In~\cite{cr-2015} and~\cite{shao} a $2$-rainbow dominating
function $f \colon V(G) \rightarrow 2^{[2]}$ is an \emph{independent $2$-rainbow dominating function} if the set
$\{x \in V(G) : f(x) \not= \emptyset\}$ is an independent subset of $G$. This induces a partition $(V_{\emptyset},V_1,V_2,V_{12})$
of $V(G)$ (in which $V_{12}=f^{-1}(\{1,2\})$) such that $V_1 \cup V_2\cup V_{12}$ is an independent dominating set of $G$, which has the additional
requirement that every vertex in $V_{\emptyset}$ either has a neighbor in $V_{12}$ or a neighbor in each of
$V_1$ and $V_2$.  The independent $2$-rainbow domination number (\cite{cr-2015}, \cite{shao}) is defined by
$i_{r2}(G)=\min\{|V_1|+|V_2|+2|V_{12}|\}$, where the minimum is computed over all partitions $(V_{\emptyset},V_1,V_2,V_{12})$ arising
from independent $2$-rainbow dominating functions of $G$. For $k \ge 2$, their independent $k$-rainbow domination number is defined
similarly.
Since $f$ is a $2$-rainbow dominating function, it does correspond to a dominating set of $G \cart K_2$,
but not always (in particular, when $V_{12}\not=\emptyset$) to an independent dominating set of $G \cart K_2$.
There are, of course, graphs $G$ such that
$i_{rk}(G)=\rik(G)$, but in general these invariants are incomparable.  For example, $\gamma_{ri3}(S_7)=6$
while $i_{r3}(S_7)=3$.  On the other hand, it is easy to verify that for the graph $G$ shown in Figure~\ref{fig:incomp},
$\gamma_{ri3}(G)=3$ and $i_{r3}(G)=4$.

\medskip
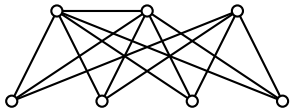
\begin{figure}[htb]
\begin{center}
\begin{tikzpicture}[scale=.6,style=thick,x=1cm,y=1cm]
\def\vr{3.5pt} 
\path (-1,2) coordinate (u0); \path (1,2) coordinate (u1); \path (3,2) coordinate (u2);
\path (-2,0) coordinate (v1); \path (0,0) coordinate (v2); \path (2,0) coordinate (v3);
\path (4,0) coordinate (v4);
\foreach \i in {1,...,4}
{  \draw (u0)--(v\i); \draw (u1)--(v\i); \draw (u2)--(v\i);}
\draw (u0)--(u1);
\foreach \i in {0,1,2}
{  \draw (u\i)  [fill=white] circle (\vr);}
\foreach \i in {1,2,3,4}
{ \draw (v\i)  [fill=white] circle (\vr);}
\end{tikzpicture}
\end{center}
\vskip -0.6 cm
\caption{$\gamma_{ri3}(G)=3$ and $i_{r3}(G)=4$} \label{fig:incomp}
\end{figure}

Our motivation for defining the new invariant, the $k$-rainbow independent domination number, is to reduce the problem of computing the
independent domination number of the generalized prism $G \cart K_k$ to an integer labeling problem on $G$.  Hence, we do not allow a vertex of
$G$ to receive a label containing more than one integer from $[k]$ nor do we require $V_1\cup \cdots \cup V_k$ to be independent.

The rest of this paper is organized as follows.  Section 2 is dedicated to basic properties and bounds concerning the $k$-rainbow independent domination. Then in Section 3 we study the $2$-rainbow independent domination number on some standard families of graphs and prove that for any non-trivial tree the $2$-rainbow independent domination number always exceeds the independent domination number. In Section 4 we establish the lower and the upper bound on the sum of  $2$-rainbow independent domination number of $G$ and of $\overline{G}$ in terms of the number of vertices of $G$, and in the last section we pose a few open problems.

\section{Basic properties and bounds}

The following observation follows directly from the definition of the $k$-rainbow independent domination
number.

\begin{ob} \label{ob:smallest}
If $G$ has order $n$ and $n \le k$, then $\rik(G)=n$.
\end{ob}

For a graph $G$ of order larger than $k$, the definition also implies that $\rik(G) \ge k$.  The following
result characterizes those graphs that have $k$-rainbow independent domination number equal to $k$.

\medskip

\begin{prop} \label{prop:lowerchar}
Let $k$ and $n$ be positive integers such that $n \ge k$.
For any connected graph $G$ of order $n$, $\rik(G)=k$ if and only if $n=k$ or $G$ has a spanning subgraph isomorphic
to $K_{k,n-k}$.
\end{prop}

\begin{proof}
Let $G$ have order $n$ and assume that $n \ge k$.  We assume throughout that $V(G)=\{x_1,\ldots,x_n\}$.
Suppose first that $n=k$.  The function $f$ defined by $f(x_j)=j$ for $j \in [k]$
is clearly a $k$RiDF of $G$.  By Observation~\ref{ob:smallest} it follows that $f$ is a $\rik(G)$-function and thus $\rik(G)=k$.
Now suppose that the order of $G$ is $n>k$  and assume that $K_{k,n-k}$
is a spanning subgraph of $G$ with partite sets $\{x_1,\ldots,x_k\}$ and $\{x_{k+1},\ldots,x_n\}$.  For $i\in [k]$, let $V_i=\{x_i\}$
and let $V_0=\{x_{k+1},\ldots,x_n\}$.  It follows that $(V_0,V_1,\ldots,V_k)$ is a $k$RiDF-partition of $G$, and again
$\rik(G)=k$.
For the converse assume that $\rik(G)=k$ and that $n>k$.  Let $(V_0,V_1,\ldots,V_k)$ be a $k$RiDF-partition of $G$.
Suppose first that $V_0=\emptyset$.  This implies that $n=|V(G)|=\sum_{j=1}^k|V_j|=\rik(G)=k$, which is a contradiction.
Therefore, $V_0 \not=\emptyset$.  Each vertex in $V_0$ has at least one neighbor in $V_j$ for
each $j \in [k]$.  Since $\rik(G)=k$,  it follows that $|V_j|=1$ for each $j \in [k]$, and consequently $G$ has a spanning subgraph
isomorphic to $K_{k,n-k}$.
\end{proof}

 Let $(V_0,V_1,\ldots,V_k)$
be a $k$RiDF-partition of $G$.  If $x$ is any vertex of $G$ and
$\deg(x)<k$, then $x \not\in V_0$.  Thus, if we let $n_i$ denote the
number of vertices in $G$ that have degree $i$, then it follows
immediately that $\rik(G)\ge \sum_{i=1}^{k-1}n_i$.  In
particular, if $G$ has maximum degree $\Delta$ and $\Delta < k$, then
$\rik(G)=|V(G)|$, which is the maximal possible value for the
$k$-rainbow independent domination number of a graph.

\begin{ob}
If $k$ is a positive integer and $G$ is a graph of order $n$ such that $\Delta(G)<k$,
then  $\rik(G)=n$.
\end{ob}

The aforementioned observation, which lead to the introduction of the $k$-rainbow independent domination concept, is the following.

\begin{prop} \label{ob:connection}
If $k$ is a positive integer and $G$  an arbitrary graph, then $\rik(G)=i(G \cart K_k)$.
\end{prop}

\begin{proof}
Let $(V_0,V_1,\ldots,V_k)$ be a $k$RiDF-partition of $G$. We
define a subset $D$ of $V(G\Box K_{k})$ by
$D=\bigcup_{i=1}^k\{(g,i) : g \in V_i\}$. It follows from the structure of the Cartesian product that
$D$ is an independent dominating set in $G \cart K_k$. Thus $\rik(G)=\left|D\right| \geq i(G \cart K_k)$.

Now assume that $D$ is a smallest independent dominating set of $G\Box K_{k}$. For $i \in [k]$, let
$V_i=\{v\in V(G): (v,i)\in D\}$, and let $V_{0}=V(G)-(V_{1}\cup \cdots \cup V_{k})$. Note that the sets $V_i$
are well defined since the independence of $D$ implies that there is at most one vertex of $D$ in each $K_k$-layer.
Thus we have $V_i\cap V_j=\emptyset$ for $0\leq i< j \leq k$ and $V_0\cup V_{1}\cup \cdots \cup V_{k}=V(G)$.
In addition, for $1\leq  i\leq  k$, $V_i$ is an independent set (otherwise we obtain a contradiction with $D$ being
independent in $G\Box K_{k}$). Since $D$ is a dominating set in $G\Box K_{k}$, we have that for every $1\leq i\leq k$
and $x\in V_0$, $(x,i)$ is adjacent to some $(g,i)\in V_i\times \{i\}\subseteq D$, which implies that
$N(x)\cap V_i \neq \emptyset$. Therefore  $(V_0,V_1,\ldots,V_k)$ represents a $k$RiDF  of $G$ and we have
$\rik(G) \leq \left|D\right|=i(G \cart K_k)$.
\end{proof}

It follows immediately that

$$\rk(G)=\gamma(G \cart K_k)\leq i(G \cart K_k)=\rik(G)\leq \alpha (G \cart K_k).$$

For any positive integer $k$, the independent domination number and the independence number of a graph $G$ have a natural
relationship to $\rik(G)$.

\begin{prop} \label{prop:bounds}
Let $k$ be a positive integer.  If $G$ is any graph, then
\[i(G) \le \rik(G) \le k\alpha(G)\,.\]
\end{prop}
\begin{proof}  The proposition is clearly true for $k=1$.  Suppose then that $k \ge 2$.  Let $(V_0,V_1, \ldots,V_k)$
be any $k$RiDF-partition of $G$.   By the definition of a $k$RiDF-partition it follows that $V_1$ dominates $V_0$
and that $V_i$ is independent for each $i \in [k]$.  Let $V_1'=V_1$.
For $2 \le i \le k$, starting with $i=2$ and proceeding through increasing values of $i$, let
$V_i'=V_i-\{ t \in V_i\,:\, N(t) \cap (V_1' \cup \cdots \cup V'_{i-1}) \not = \emptyset \}$.  Define
$W$ by $W=V_1'  \cup \cdots \cup V_k'$.   By its construction we see that $W$ is independent.  Let $x\in V(G)-W$.
If $x \in V_0$, then $x$ has a neighbor in $V_1$ (that is, a neighbor in $V_1'$) and hence $W$ dominates $x$.
If $x \in V_j - W$ for some $j\ge 2$,
then $x \in \{ t \in V_j\,:\, N(t) \cap (V_1' \cup \cdots \cup V'_{j-1}) \not = \emptyset \}$,
and consequently $x$ is dominated by $V_1' \cup \cdots \cup V'_{j-1}$ and hence also by $W$. This proves that $W$
is a dominating set of $G$.  We have shown that $W$ is an independent dominating set of $G$.  Furthermore,
for each $i \in [k]$, $V_i$ is independent.  This implies
\[i(G)\le |W| =\sum_{i=1}^k|V_i'| \le \rik(G)=\sum_{i=1}^k|V_i|\le k\alpha(G)\,.\]
\end{proof}

There are graphs $G$ that show the lower bound in Proposition~\ref{prop:bounds} is sharp.  For example, for any positive
integer $k \ge 2$ and any positive integer $m \ge 2$, let
$G$ be the complete multipartite graph $K_{k,n_2,\ldots,n_m}$ where $k \le n_2 \le \cdots \le n_m$.  Let
$\{a_1,\ldots,a_k\}$ be the partite set in $G$ of size $k$.  For each $i\in[k]$ let $V_i=\{a_i\}$ and let
$V_0=V(G)-\{a_1,\ldots,a_k\}$.  It is now easy to verify that $(V_0,V_1, \ldots,V_k)$ is a $k$RiDF-partition
of $G$ and that $i(G)=k=\rik(G)$.

Another class of graphs that attain the lower bound in Proposition~\ref{prop:bounds} for $k=2$ were studied
by Hartnell and Rall~\cite{hr-2004}.  They gave a necessary and sufficient condition on a graph $G$ for
$\gamma(G)=\rd(G)$.  Such a graph has a minimum dominating set $D$ that can be partitioned as $D_1 \cup D_2$ such
that $V(G)-N[D_1]=D_2$ and $V(G)-N[D_2]=D_1$.  In addition, each of $D_1$ and $D_2$ is a $2$-packing (i.e. a subset of $V(G)$ in which all the vertices are in distance at least $3$ from each other) and $D$ is
independent.  This implies that $i(G)=\gamma(G)$ and gives  a $2$RiDF-partition $(V(G)-D,D_1,D_2)$ of $G$.
Two examples of this class of graphs are shown in Figure~\ref{fig:example1}. For $i\in [2]$ the set $D_i$ consists of the vertices
labeled $i$.

\begin{figure}[htb]
\begin{center}
\begin{minipage}{0.30\textwidth}
\begin{tikzpicture}[scale=.6,style=thick,x=1cm,y=1cm]
\def\vr{3.5pt} 
\path (-2,2) coordinate (u0); \path (-1,2) coordinate (u1); \path (0,2) coordinate (u2);
\path (1,2) coordinate (u3); \path (2,2) coordinate (u4); \path (3,2) coordinate (u5);
\path (-2,-2) coordinate (v0); \path (-1,-2) coordinate (v1); \path (0,-2) coordinate (v2);
\path (1,-2) coordinate (v3); \path (2,-2) coordinate (v4); \path (3,-2) coordinate (v5);
\foreach \i in {2,...,5}
{  \draw (u0)--(v\i); \draw (u1)--(v\i);}
\foreach \i in {0,1,4,5}
{  \draw (u2)--(v\i); \draw (u3)--(v\i);}
\foreach \i in {0,1,2,3}
{  \draw (u4)--(v\i); \draw (u5)--(v\i);}
\foreach \i in {0,...,5}
{  \draw (u\i)  [fill=white] circle (\vr); \draw (v\i)  [fill=white] circle (\vr);}
\draw[anchor = south] (u0) node {$1$}; \draw[anchor = south] (u1) node {$2$};
\draw[anchor = north] (v0) node {$1$}; \draw[anchor = north] (v1) node {$2$};
\end{tikzpicture}
\end{minipage}%
\begin{minipage}{0.30\textwidth}
\begin{tikzpicture}[scale=.6,style=thick,x=1cm,y=1cm]
\def\vr{3.5pt} 
\path (-3,0) coordinate (u0); \path (-1.5,2) coordinate (u1); \path (-1.5,1) coordinate (u2);
\path (-1.5,-1) coordinate (u3); \path (-1.5,-2) coordinate (u4); \path (0,0) coordinate (u5);
\path (1.5,1) coordinate (u6); \path (1.5,-1) coordinate (u7); \path (3,0) coordinate (u8);
\foreach \i in {1,...,4}
{  \draw (u0) -- (u\i);  \draw (u5) -- (u\i);}
\foreach \i in {6,7}
{  \draw (u8) -- (u\i);  \draw (u5) -- (u\i);}
\foreach \i in {0,...,8}
{  \draw (u\i)  [fill=white] circle (\vr); }
\draw[anchor = east] (u0) node {$1$}; \draw[anchor = north] (u5) node {$2$}; \draw[anchor = west] (u8) node {$1$};
\end{tikzpicture}
\end{minipage}%
\end{center}
\vskip -0.6 cm
\caption{Examples with $i=\ri$.} \label{fig:example1}
\end{figure}
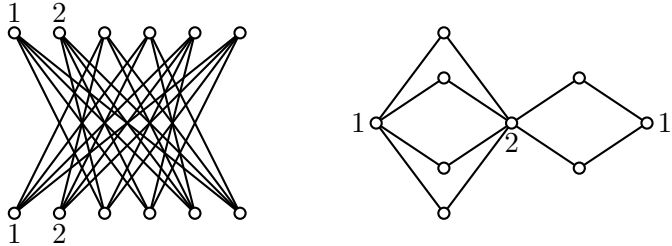

By examining the proof of Proposition~\ref{prop:bounds} we can see
that more can be said about any graph that attains the lower bound.  Indeed, for such a graph $V_i'=V_i$
for each $i \in [k]$.  This gives the following corollary.

\begin{cor} \label{independent}
Let $k$ be a positive integer.  If $G$ is a graph such that $i(G)=\rik(G)$, then for any $k$RiDF-partition
$(V_0,V_1, \ldots,V_k)$ of $G$, the set $V_1\cup \cdots \cup V_k$ is independent.
\end{cor}

We will have use of a concept first introduced by
Fradkin~\cite{f2009} who defined a {\it greedy independent
decomposition} (GID)  of a graph $G$ to be a partition $A_1,\ldots,
A_t$ of $V(G)$ such that $A_1$ is a maximal independent set in $G$,
and for each $2 \le i \le t$, the set $A_i$ is a maximal independent set in the
graph $G-(A_1\cup \cdots\cup A_{i-1})$.  The number of subsets in the partition, here $t$,
is called its {\it length}.  In what follows we do not require the last subset $A_t$ in the sequence to be non-empty and
when it is non-empty we do not require it to be independent.  With these modifications
we call such a sequence of subsets of $V(G)$  a {\it partial GID} of $G$.

Note that if $A_1,\ldots,A_t$ is a partial GID of $G$, then each of
$A_1, \ldots,A_{t-1}$ is an independent set of $G$,
and for each $i$ such that $1 \le i \le t-1$, $A_i$ dominates $A_t$.
This proves the correctness of the upper bound in the following
proposition. We choose a labeling of the parts of the partition to
fit with our notation.

\begin{prop} \label{prop:pGID}
 If $V_1,V_2,\ldots,V_k,V_0$ is a partial GID of $G$, then
\[\rik(G)\leq \sum_{i=1}^k|V_i|\,.\]
\end{prop}

If $G=K_n$ for any $n \ge k$, then by noting as  above that $\rik(G)=i(G \cp K_k)$,
it is easy to show that the bound in Proposition~\ref{prop:pGID} is sharp. That is, if
$n \ge k$, then $\rik(K_n)=i(K_n \cart K_k)=\min\{n,k\}=k$.  Indeed, in this case a complete graph of order
$n$ has a partial GID of length $k+1$, and all of the independent parts of the decomposition have order
1. However, if $k >n$, then $K_n$ does not have a partial GID of length $k+1$, and thus
Proposition~\ref{prop:pGID} is not applicable.  In this case, $\rik(K_n)=n$.

\section{Some families of graphs and their $2$RiDF}

In this section we first give the value of the invariant $\ri$ on some  common classes of
graphs, and then we prove that $i(T)< \ri(T)$ for any nontrivial tree $T$.

Jacobson and Kinch~\cite{jk-1983} determined that for the path $P_n$ on $n$ vertices, $\gamma(P_n \cart K_2)=\left\lceil \frac{n+1}{2}
\right\rceil$, and Cort\'{e}s~\cite{cor} proved that $i\left( P_{n} \cart K_2 \right)=\left\lceil \frac{n+1}{2} \right\rceil$.  In the
language of rainbow domination this is equivalent to
\[ \gamma _{r2}\left( P_{n}\right) =\left\lceil \frac{n+1}{2} \right\rceil= \ri\left( P_{n}\right)\,,\]
(see \cite{bks-2007} and Proposition~\ref{ob:connection}).  Pavli\v{c} and \v{Z}erovnik~\cite{pz-2012} proved
that $\ri\left( C_{n}\right)= \left\lceil n/2 \right\rceil$ for $n$ congruent to either 0 or 3 modulo 4, while
$\ri\left( C_{n}\right)= \left\lceil n/2 \right\rceil+1$ for $n$ congruent to 1 or 2 modulo 4.  In addition,
the following values are easy to verify.
\begin{itemize}
\item For a star $S_n$ with $n \ge 3$, $\ri\left( S_{n}\right)=n-1$.
\item For $n\ge 2$, $\ri\left( K_{n}\right)=2$.
\item For a complete multipartite graph $K_{r_1,\ldots,r_k}$, where $2\leq r_1 \leq \ldots\leq r_k$,
$\ri\left(K_{r_1,\ldots,r_k} \right)=r_1=i\left(K_{r_1,\ldots,r_k}\right)$.
\end{itemize}

From Proposition \ref{prop:bounds} it follows that $i(G)\leq \ri(G)$. In the following we show that for trees we have strict inequality.
 \begin{lem}
\label{lemleaf}
If $x$ is a leaf of a tree $T$, then $i(T)=i(T-x)$ or $i(T)= i(T-x)+1$.
\end{lem}

\begin{proof}
Let $x$ be a leaf of a tree $T$ and $y$ its unique neighbor. For notational convenience let $T'=T-x$.  To prove the lemma we show
that $i(T') \leq i(T) \leq i(T')+1$. Let $M$ be an independent dominating set of $T$ of cardinality $i(T)$.
If $x \notin M$, then $M$ dominates $T'$ and is independent in $T'$, which implies $i(T')\leq \left|M\right|=i(T)$. Now suppose that $x \in M$.
If $M \setminus \{x\}$ dominates $T'$, then  $i(T')\leq \left|M\right|-1<i(T)$.  If  $M \setminus \{x\}$ does
not dominate $T'$, then  $N(y)\cap M=\{x\}$ and  $M'=\left(M \setminus \left\{x\right\}\right)\cup\left\{y\right\}$
is an independent dominating set of $T'$, which shows  $i(T')\leq \left| M'\right|=\left| M\right|=i(T)$.  This establishes the first inequality.
To prove the second inequality, let $J$ be an independent dominating set of $T'$ such that $i(T')=|J|$.  If $y \in J$, then $J$ also dominates
$T$ and thus $i(T) \leq |J|=i(T')$.  If $y \notin J$, then $J \cup \{x\}$ is an independent dominating set
of $T$, which implies $i(T) \leq |J|+1=i(T')+1$.  Therefore, $i(T-x) \leq i(T) \leq i(T-x)+1$.
\end{proof}

\begin{lem}\label{prop-tree2}
Let $x$ be a leaf of a nontrivial tree $T$.  If $T'=T-x$, then $\ri(T)=\ri(T')$ or
$\ri(T)=\ri(T')+1$.
\end{lem}

\begin{proof}
Let $T$, $x$ and $T'$ be as in the statement of the lemma, and let $y$ be the unique neighbor of $x$ in $T$.
Recall from Proposition~\ref{ob:connection} that $\ri(G)=i(G \cart K_2)$ for any graph $G$.  We will prove the lemma
by showing that $i(T\Box K_2)=i(T'\Box K_2)$ or $i(T\Box K_2)=i(T'\Box K_2)+1$.  First we shall prove
that $i(T\Box K_2)\leq i(T'\Box K_2)+1$. Let $J$ be an independent dominating set of  $T'\Box K_2$ such that $|J|=i(T'\Box K_2)$.  Since
$(y,1)$ and $(y,2)$ are adjacent in $T'\Box K_2$, at least one of them, say $(y,2)$, is not in $J$. The set $J\cup \left\{(x,2)\right\}$ is independent
in $T\Box K_2$ and dominates $T\Box K_2$. Hence $i(T\Box K_2)\leq |J|+1=i(T'\Box K_2)+1$.

Now we shall prove that $i(T'\Box K_2)\leq i(T\Box K_2)$. Let $M$ be an independent dominating set of $T\Box K_2$ of cardinality $i(T\Box K_2)$.
Since $\left\{(y,1), (y,2)\right\}$ is not a subset of $M$, exactly one of $(x,1)$ or $(x,2)$ is in $M$. Without loss of generality we assume $(x,1)\in M$.
If $(y,2)\in M$, then $M \setminus \left\{(x,1)\right\}$ is an independent dominating set of $T'\Box K_2$ and thus $i(T'\Box K_2)\leq|M|-1<i(T\Box K_2)$.
If $(y,2) \notin M$, then there exists $(z,2)\in N_{T'\Box K_2}(y,2)\cap M$.  If there also exists  $(w,1)\in N_{T'\Box K_2}(y,1)\cap M$, then
$M\setminus\left\{(x,1)\right\}$ is an independent dominating set of $T'\Box K_2$ and again $i(T'\Box K_2)\leq|M|-1<i(T\Box K_2)$. On the other hand,
if  $N_{T'\Box K_2}(y,1)\cap M=\emptyset$, then $\left(M\setminus\left\{(x,1)\right\}\right)\cup \left\{(y,1)\right\}$ is an
independent dominating set of $T'\Box K_2$, which implies  $i(T'\Box K_2)\leq|M|=i(T\Box K_2)$, and the proof is complete.
\end{proof}

\begin{ob}\label{leaf}
If $G$ is a graph without isolated vertices and $(V_0,V_1,V_2)$ is a $2$RiDF-partition of $G$, then every leaf of $G$ belongs to $V_1 \cup V_2$.
\end{ob}

The above results help us to show that for any non-trivial tree $T$ the   $2$-rainbow independent domination
number  always exceeds the independent domination number.

\begin{thm}
If $T$ is a non-trivial tree, then $i(T)<\ri(T)$.
\end{thm}

\begin{proof}
The proof is by induction on the order of the tree. It is straightfoward to show the statement holds for all trees of
order  2, 3, or 4.  Let $n$ be an integer such that $n\ge 5$ and assume the theorem is true for all trees of order less than $n$.
Let $T$ be a tree of order  $n$, let $x$
be a vertex of degree $1$ in $T$ and let $y$ be the unique neighbor of $x$. Let $T'$ denote the subtree $T-x$ of $T$.
By the induction hypothesis $i(T')<\ri\left( T'\right)$. Using Lemma~\ref{lemleaf} and Lemma~\ref{prop-tree2} we get four cases.\\
Case 1.  $i(T)=i(T')$ and $\ri\left(T\right)=\ri\left(T'\right)$. \\
Case 2.  $i(T)=i(T')+1$ and $\ri\left(T\right)=\ri\left(T'\right)+1$. \\
Case 3.  $i(T)=i(T')$ and $\ri\left(T\right)=\ri\left(T'\right)+1$. \\
Case 4.  $i(T)=i(T')+1$ and $\ri\left(T\right)=\ri\left(T'\right)$. \\
It is easy to observe and is left to the reader to show  that in the first three cases  $i(T)<\ri\left( T\right)$.
Just using the equations in Case 4 it could happen that $i(T)=\ri\left( T\right)$.  We now show that this leads to a contradiction.
Let $f:V(T) \to \{0,1,2\}$  be a $\ri$-function and $(V_0,V_1,V_2)$ the resulting $2$RiDF-partition of $T$.  By Observation~\ref{leaf}, $x \in V_1 \cup V_2$
and by Corollary~\ref{independent} the set $V_1\cup V_2$ is independent. Without loss of generality we may assume that $x \in V_1$.
This implies that $y \in V_0$ and also that there exists a vertex $z\in V_2 \cap N_T(y)$.    Let $I=\{y\} \cup \{t \in V_1 \cup V_2 : d(y,t) \ge 2\}$.
Clearly, $I$ is independent and $|I|<|V_1 \cup V_2|=i(T)$.  We claim that $I$ dominates $T$.  By the definition of $I$, it suffices to show that $I$ dominates
the set $\{t\in V(T) : d(y,t)=2\}$.  Let $w$ be a vertex such that $d(y,w)=2$ and $w \notin I$.  Thus, $w \in V_0$.  Since $(V_0,V_1,V_2)$  is a $2$RiDF-partition
of $T$, $w$ has a neighbor $a$ in $V_1$ and a neighbor $b$ in $V_2$.  At most one of $a$ or $b$ is in $N_T(y)$ and it follows that $w$ is dominated by $I$.
This is a contradiction, and thus in all cases, $i(T)<\ri\left( T\right)$.  The theorem follows by induction.
\end{proof}

\section{A Nordhaus-Gaddum type of result for $\ri(G)$}

In 1956, Nordhaus and Gaddum~\cite{NH-1956} gave lower and upper bounds on the sum and the product of
the chromatic number of a graph and its complement, in terms of the order of the graph~\cite{NH-1956}. Relations of a similar type have been proposed
for many other graph invariants since then.  See the survey~\cite{NG-survey} by Aouchiche and Hansen on this topic.
An invariant not addressed in the survey was rainbow domination.  Wu and Xing~\cite{wx-2010} proved sharp Nordhaus-Gaddum type
bounds for the $2$-rainbow domination number.  In particular, they proved that if $G$ is any graph of order $n$ and $n \ge 3$, then $5 \le \gamma_{r2}(G)+\gamma_{r2}(\overline{G}) \le n+2$.

In this section we explore this relation for the $2$-rainbow independent domination number. Although our bounds are very
similar to those given by Wu and Xing for (ordinary) $2$-rainbow domination, the proof of our bounds turned out to be more involved
than theirs because of the requirement that $V_1$ and $V_2$ are independent sets in
the definition of a $2$RiDF-partition of $G$.  We will need the following lemmas to prove our main result Theorem~\ref{thm:sum}.

\begin{lem} \label{un} For any graph $G$ of order $n$,
$\ri(G)=n$ if and only if every connected component of $G$ is isomorphic either to $K_1$ or $K_2$. In addition, if
$\ri(G)=n$, then $\ri(\overline{G})=2$.
\end{lem}

\begin{proof}  If every connected component of a graph $G$ of order $n$  is isomorphic to $K_1$ or $K_2$, then clearly $\ri(G)=n$.
Since $\ri(G)=\sum_{j=1}^l\ri(G_j)$ if $G_1,G_2,\ldots,G_l$ are the connected components of $G$, then to prove the
converse it suffices to show that $\ri(G)<n$ when $G$ is connected and $n \ge 3$. Thus, with the goal of
obtaining a contradiction suppose that $G$ is connected of order $n\ge 3$ and that $\ri(G)=n$.
There exists $y\in V(G)$ of degree at least $2$.
Let $f$ be a $\ri(G)$-function and let $(V_0,V_1,V_2)$ be the $2$RiDF-partition associated with $f$. Since $\ri(G)=n$,
we have $V_0=\emptyset$, and since for each $i\in [2]$,
there is no edge between two vertices from $V_i$, $G$ is a bipartite graph where $V_1$ and $V_2$ constitute a bipartition of $V(G)$. Assume
(without loss of generality) that $y\in V_1$. Then there exist $x,z\in N(y)\cap V_2$. Let $D_i=\{v\in V(G): d(v,y)=i\}$ and suppose the
eccentricity of $y$ (i.e., the largest distance of $y$ to a vertex from $V(G)$) is $t$. Then $\{y\}\cup\bigcup_{i=1}^t D_i=V(G)$. Note
that $E(G\langle D_i\rangle)=\emptyset$, and if $v\in D_i$, then $f(v)=1$ in the case when $i$ is even and $f(v)=2$  when $i$
is odd. We will show that a $2$RiDF, $g:V(G)\rightarrow \{0,1,2\}$ with $w(g)<w(f)$ can be constructed as follows.

Let $g(y)=0$, $g(x)=1$ and $g(z)=2$. Let $A=D_2 \cap (N(x)\setminus N(z))$, $B=D_2\cap (N(z)\setminus N(x))$, and $C=D_2 \cap N(x)\cap N(z)$.
Set $g(v)=2$ if $v\in A$,  $g(v)=1$ if $v\in B$, and  $g(v)=0$ if $v\in C$. Next, we define the following sets: $F$ is the set of vertices
$v\in D_1\setminus\{x,z\}$ such that $v$ has a neighbor in $A$ as well as a neighbor in $B$; $D$ is the set of vertices $v\in D_1\setminus\{x,z\}$
such that $v$ has a neighbor in $B$, but no neighbor in $A$; and $E$ is the set of vertices $v\in D_1\setminus\{x,z\}$ such that $v$ has a
neighbor in $A$, but no neighbor in $B$. Set $g(v)=2$ if $v\in D$,  $g(v)=1$ if $v\in E$, and  $g(v)=0$ if $v\in F$.
Now consider vertices in $D_1\setminus (\{x,z\}\cup D\cup E\cup F)$. None of these vertices is adjacent to a vertex in $A$. Hence we let $g(v)=2$
for every $v\in D_1\setminus (\{x,z\}\cup D\cup E\cup F)$. Next, we define $g$ on $D_2\setminus (A\cup B\cup C)$ according to the following procedure.

Suppose we have already assigned values of $g$ to vertices of a set $D_i$, $i\in \{1,\ldots, t-1\}$. Let $M=\{v\in D_i: g(v)=1\}$, $N=\{v\in D_i: g(v)=2\}$,
and $O=\{v\in D_i: g(v)=0\}$. Now consider vertices $w$ from $D_{i+1}$. If $w$ has a neighbor in $M$ and a neighbor in $N$, then $g(w)=0$. If $w$
has a neighbor in $M$, but not in $N$,  $g(w)=2$. If $w$ has a neighbor in $N$, but not in  $M$, $g(w)=1$. For the remaining vertices in $D_{i+1}$
 we set $g(w)=f(w)$. In this way we first assign values of $g$ to all vertices of $D_2\setminus (A\cup B\cup C)$, and then to all remaining vertices in $V(G)$.

It is straightforward to check that $g$ is a $2$RiDF on $G$. Since at least one vertex (that is $y$) has the value $0$ under $g$, we
have $w(g)\leq n-1$, as desired.

For the proof of the second part of the lemma, note that if $G$ is an empty graph on $n$ vertices, then $\overline{G}=K_n$ and
$\ri(\overline{G})=2$. If $\ri(G)=n$ and at least two vertices, say $a$ and $b$, are adjacent in $G$, then
$N_{\overline{G}}(a)=V(G)-\{a,b\}$ and $N_{\overline{G}}(b)=V(G)-\{a,b\}$.  Thus $h:V(G)\rightarrow \{0,1,2\}$, $h(a)=1$, $h(b)=2$ and
$h(c)=0$ otherwise, is a $2$RiDF with  minimum weight.
\end{proof}

\begin{lem} \label{star} If $G$ is a graph of order $n\geq3$ containing a universal vertex and $\ri(G)=n-1$, then $G$ is either
isomorphic to $S_n$ or $S^+_n$.
\end{lem}

\begin{proof} Let $u$ be a universal vertex of $G$ and assume $\ri(G)=n-1$. We claim there is a $\ri(G)$-function $f$ for
which $f(u)=0$.  For suppose that  $(W_0,W_1,W_2)$ is a $2$RiDF-partition corresponding to a $\ri(G)$-function, and
$W_0=\{v\}$ for some $v \not=u$.  Without loss of generality we may assume that $u \in W_1$.  Since $W_1$ is independent and $u$ is a universal
vertex, it follows that $W_1=\{u\}$.  If $|N(v) \cap N(u)|=1$, then $G=S_n^+$.  Otherwise, let $r$ and $s$ be distinct vertices in
$W_2 \cap N(v)$.  Now, let $W_0'=\{u,v\}, W_1'=\{s\}$ and $W_2'=W_2-\{s\}$. We see that $(W_0',W_1',W_2')$ is a $2$RiDF-partition of $G$, which
contradicts our assumption that $\ri(G)=n-1$.  Thus, we may assume that such a $\ri(G)$-function $f$ exists, and let $(V_0,V_1,V_2)$
be the partition of $V(G)$ corresponding to $f$. Assume that $G\langle N(u)\rangle$ contains at least two edges. In the first case,
let $xy$ and $yz$ be edges in $G\langle N(u)\rangle$, and suppose (without loss of generality) that $f(x)=f(z)=1$.  Then  $f(y)=2$.
We will consider a function $g:V(G)\rightarrow \{0,1,2\}$, defined as follows. Let $A=V_2-N_G(z)$, let $B$ be the set of vertices in $V_2\cap N(z)$
having another neighbor in $V_1$ different from $z$, and let $C$  be the set of vertices in $V_2$ whose only neighbor in $V_1$ is $z$. Let $g(v)=2$ for
$v\in A\cup\{z\}$, $g(v)=0$ for $v\in B\cup\{u\}$, and $g(v)=1$ for $v\in C\cup (V_1-\{z\})$.
The function $g$ is a $2$RiDF with $w(g)\leq n-2$ (since $g(u)=0$ and $g(y)=0$) a contradiction. This implies that any two edges from
$G\langle N(u)\rangle$ are disjoint. Assume that $G\langle N(u)\rangle$ contains two such edges,  $xy$ and $ab$, and suppose
$f(y)=f(a)=1$ and $f(x)=f(b)=2$. Let $g:V(G)\rightarrow \{0,1,2\}$
be defined by $g(u)=1$, $g(v)=0$ for every $v\in V_1$ that has a neighbor in $V_2$, $g(v)=2$ for every $v\in V_1$ that does not have a neighbor in $V_2$, and
$g(t)=f(t)$ otherwise.  This function $g$ is a $2$RiDF, and since $g(y)=g(a)=0$, we obtain $w(g)\leq n-2$, a contradiction again with $\ri(G)=n-1$.
This implies that $G\langle N(u)\rangle$ contains at most one edge, which means that $G$ is either isomorphic to $S_n$ or to $S^+_n$.
\end{proof}

\begin{thm} \label{thm:sum} If $G$ is a graph of order $n$ where $n \ge 3$, then $$5\leq \ri(G)+\ri(\overline{G})\leq n+3,$$
and the bounds are sharp.
\end{thm}
\begin{proof}
Let $G$ be a graph of order $n$ such that $n \ge 3$.  First we consider the lower bound. Clearly,
$\ri(H)\geq 2$ for any graph $H$ of order at least 2, and thus  $\ri(G)+\ri(\overline{G})\geq 4$.
However, this bound cannot be achieved. Namely, if $\ri(G)= 2$ and $(V_0,V_1,V_2)$ corresponds to a $\ri(G)$-function of $G$, where
$V_1=\{x\}$ and $V_2=\{y\}$, then in $\overline{G}$ the set $\{x,y\}$ induces either $N_2$ or $K_2$, and none of the vertices
from $V_0$ is adjacent to $x$ or to $y$. Thus we have $\ri(\overline{G})\geq 3$, and the lower bound follows.

For the upper bound, let a partition $(V_0,V_1,V_2)$ of $V(G)$ correspond to a $\ri(G)$-function $f$ of $G$ (then $\{V_0,V_1,V_2\}$ is a
set partition of $V(\overline{G})$). Thus $\ri(G)=n-\left|V_0\right|$. If $\ri(G)=2$ then clearly
$\ri(G)+\ri(\overline{G})\leq n+2$.
Now consider the case when at least one of $V_1$ and $V_2$, say $V_1$, contains at least two vertices.
Let $x$ and $y$ be different vertices in $V_1$, which implies that $xy\in E(\overline{G})$. Let $A$  be a maximal independent set in
$\overline{G}$ containing $x$, and $B$ a maximal independent set in $\overline{G}-A$ containing $y$. Since $A,B,V(\overline{G})-(A\cup B)$
is a partial GID of $\overline{G}$, Propositon~\ref{prop:pGID} implies that $\ri(\overline{G})\leq \left|A \right|+ \left|B \right|$.
Since $V_1$ induces a complete subgraph in $\overline{G}$, we have $(A\cup B)\cap V_1=\{x,y\}$.  Similarly,  $V_2$ induces a complete
subgraph in $\overline{G}$, which implies there is at most one vertex in $A\cap V_2$ and at most one vertex in $B\cap V_2$. Thus
$$\ri(\overline{G})\leq \left|A \right|+ \left|B \right|\leq \left|V_0\right|+4,$$
which implies
\begin{equation}\label{n+4} \ri(G)+\ri(\overline{G})\leq n-\left|V_0 \right|+  \left|V_0\right|+4=n+4.\end{equation}

Now suppose that $\ri(G)+\ri(\overline{G})=n+4$ which is by (\ref{n+4}) equivalent to the fact that $\ri(\overline{G})= \left|V_0\right|+4$. The latter holds if and only if, for $A$ and $B$ defined as above, the function
$g:V(\overline{G})\rightarrow \{0,1,2\}$, $g(a)=1$ for every $a\in A$, $g(a)=2$ for every $a\in B$, and $g(a)=0$ otherwise, is a $\ri(\overline{G})$-function and $\left|V_2\cap A\right|=\left|V_2\cap B\right|=1$.  Let $s$ and $t$ be vertices in $V_2$ such that $s\in A$ and $t\in B$.  This means that for every $v\in V_0$, $g(v)\neq 0$. (Note that $V_0$ is nonempty for  otherwise
$\ri(G)=n$, which by Lemma~\ref{un} implies that $\ri(\overline{G})=2$ and leads to a contradiction.)
Also note that $\ri(\overline{G}\langle V_0\rangle)=\left|V_0\right|$, for otherwise we obtain a contradiction with $g$ being a $\ri(\overline{G})$-function. Thus  every connected component of $\overline{G}\langle V_0\rangle$  is either $K_1$ or $K_2$, by Lemma~\ref{un}.

Suppose first that $\overline{G}\langle V_0\rangle$ contains at least one edge.
We distinguish the following cases.

{\bf Case 1.} Suppose there exist $u,v\in V_0$ such that $uv\in E(\overline{G})$ and $N_{\overline{G}}(\{u,v\})-\{u,v\}=\emptyset$. Then $N_G(u)=V(G)-\{u,v\}$ and $N_G(v)=V(G)-\{u,v\}$, which implies that $\ri(G)=2$, thus $\ri(G)+\ri(\overline{G})\leq n+2$, a contradiction.

{\bf Case 2.} Thus for each pair of vertices  $u,v\in V_0$ such that $uv\in E(\overline{G})$ it follows $N_{\overline{G}}(\{u,v\})-\{u,v\}\neq \emptyset$.
Fix such an edge $uv$ in $\overline{G}\langle V_0\rangle$.
There exists an edge $uz\in E(\overline{G})$ such that $z\in V_1\cup V_2$.

First we claim that there is at most one such edge $uz$.
Suppose to the contrary that there are different vertices $z_1,z_2\in V_1\cup V_2$ such that $uz_1, uz_2\in E(\overline{G})$.
There are two possibilities, either $z_1$ and $z_2$ belong to the same one of these two sets $V_1$ and $V_2$ or they do not.

In either case we will reach a contradiction by constructing an $2$-rainbow independent
dominating function $g':V(\overline{G})\rightarrow \{0,1,2\}$ with  weight strictly less than $\left|V_0\right|+4$.  In both cases we first let $g'(u)=0$, $g'(z_1)=1$ and  $g'(z_2)=2$; next we assign values of $g'$ to all vertices from $V_1\cup V_2-\{z_1,z_2\}$; and lastly to all vertices in $V_0-\{u\}$.

Now we describe the procedure of assigning values to vertices in $V_1\cup V_2-\{z_1,z_2\}$.  First assume that $z_1,z_2$ belong to different sets $V_1$ and $V_2$, say $z_1\in V_1$ and $z_2\in V_2$. Consider vertices from $V_2-\{z_2\}$. If all of them are adjacent to $z_1$, then their value under $g'$ will be $0$. If there exist vertices in $V_2-\{z_2\}$ not adjacent to $z_1$, then one of them obtains the value $1$ and all others are assigned the value $0$. Analogously, we assign values of $g'$ to vertices in $V_1-\{z_1\}$.
Now assume that $z_1,z_2$ belong to the same set, say $V_1$. If there exist vertices $a,b\in V_2$ such that $az_1,bz_2\notin E(\overline{G})$, then $g'(a)=1$, $g'(b)=2$, and $g'(z)=0$ for every $z\in V_1\cup V_2-\{z_1,z_2,a,b\}$.
If all vertices in $V_2$ are adjacent to $z_2$ and there exist vertices in $V_2$ not adjacent to $z_1$, then one of them, say $a$, obtains value $1$ and all other vertices in $V_2-\{a\}$ as well as all vertices in $V_1-\{z_1,z_2\}$ obtain value $0$.  In an analogous manner, if all vertices in $V_2$ are adjacent to $z_1$ and there exist vertices in $V_2$ not adjacent to $z_2$, then one of them, say $b$, obtains value $2$ and all other vertices in $V_2-\{b\}$ as well as all vertices in $V_1-\{z_1,z_2\}$ obtain value $0$.
In the last case, when all vertices from $V_2$ are adjacent to both $z_1$ and $z_2$, we assign value $0$ to every vertex in $V_1\cup V_2-\{z_1,z_2\}$.  Note that in all cases at most four vertices in $V_1 \cup V_2$ have been
assigned a nonzero value under $g'$.

Now we assign values of $g'$ to all vertices in $V_0-\{u\}$. In the beginning let $C=\{z\in V_1\cup V_2: g'(z)=1\}$ and  $D=\{z\in V_1\cup V_2: g'(z)=2\}$. During the procedure  other vertices from $V_0$ may be added to the sets $C$ and $D$.  Order the vertices in $V_0-\{u\}$ arbitrarily.
By proceeding through this list we assign values under $g'$, one vertex $z$ at a time, in the following way.  If $z$  has a neighbor in $C$ as well as in $D$, then $g'(z)=0$. If $z$ has a neighbor in $D$ but not in $C$, then $g'(z)=1$ and we add $z$ to $C$.  If $z$ has a neighbor in $C$ but not in $D$, then $g'(z)=2$ and we add $z$ to $D$. If $z$ does not have a neighbor in $C \cup D$, then $g'(z)=1$ and we add $z$ to $C$.

The function $g'$ is a $2$RiDF of $\overline{G}$ and $w(g')< \left|V_0\right|+4$.  This proves the claim that if $u \in V_0$ and
$N_{\overline{G}}(u) \cap V_0 \not = \emptyset$, then $u$ has at most one neighbor in $V_1\cup V_2$.

Suppose there exist $u,v\in V_0$ such that $uv\in E(\overline{G})$ and there exist $z,w\in V_1\cup V_2$ such that $uz,vw\in E(\overline{G})$.  Observe that  $N_G(u)=V(G)-\{z,u,v\}$  and   $N_G(v)=V(G)-\{w,u,v\}$. Assume that $z\neq w$. One can easily verify that $f':V(G)\rightarrow \{0,1,2\}$, defined by $f'(u)=f'(z)=1,f'(v)=f'(w)=2$, and $f'(a)=0$ otherwise, is a $2$RiDF of $G$. This implies $\ri(G)\leq 4$, hence $\ri(\overline{G})\geq n$, which in fact gives  $\ri(\overline{G})= n$. By Lemma~\ref{un},  $\ri(G)= 2$, which leads to a contradiction. On the other hand, if $z=w$, then $f':V(G)\rightarrow \{0,1,2\}$, defined by $f'(u)=f'(z)=1,f'(v)=2$, and $f'(a)=0$ otherwise, is a $2$RiDF of $G$.  Hence $\ri(G)\leq 3$, and $\ri(G)+\ri(\overline{G})\leq n+3$, which is again a contradiction.

In the remaining case, suppose there exist  vertices $u$ and $v$ in $V_0$   such that $uv\in E(\overline{G})$ and exactly one of $u$ and $v$ is adjacent
in $\overline{G}$ to a vertex in $V_1\cup V_2$.  Without loss of generality we assume $uz\in E(\overline{G})$  and $z\in V_1\cup V_2$.
The function   $f':V(G)\rightarrow \{0,1,2\}$, defined by $f'(u)=f'(z)=1,f'(v)=2$, and $f'(a)=0$ otherwise, is a $2$RiDF of $G$. This
 implies $\ri(G)\leq 3$,  which is a contradiction.

The above two cases lead us to the conclusion that each component of $\overline{G}\langle V_0\rangle$  is an isolated vertex.  That is,
$V_0$ induces a complete subgraph in $G$.  Here the following possibilities arise.

{\bf Case a.} Assume that $N_{\overline{G}}(u)=\emptyset$, for every $u\in V_0$. If there are at least two vertices in $V_0$,
then $G$ has at least  two  universal vertices.  Thus
$\ri(G)=2$, which is a contradiction since $\ri(\overline{G})\leq n$.
On the other hand, if $V_0=\{u \}$, then $u$ is a universal vertex in $G$ and $\ri(G)=n-1$. Hence $G$ is either $S_n$ or
$S^+_n$ by Lemma~\ref{star}.  It follows that $\overline{G}-u$ has at least  two universal vertices since $V_1$ and $V_2$ both contain at least two vertices, which implies $\ri(\overline{G})\leq 3$. We conclude that  $\ri(G)+\ri(\overline{G})\leq n-1+3=n+2$,
which is a contradiction.

{\bf Case b.} In the last case assume that there exists an edge in $\overline{G}$ joining a vertex in
$V_0$ and a vertex in  $V_1\cup V_2$.

First we claim that each vertex in $V_0$ has at most one neighbor in $V_1\cup V_2$. Suppose to the contrary that for some
$u \in V_0$ there are distinct vertices $v_1,v_2\in V_1\cup V_2$ such that $uv_1, uv_2\in E(\overline{G})$.

In the two possibilities that arise we will construct an $2$-rainbow independent
dominating function $g':V(\overline{G})\rightarrow \{0,1,2\}$ with weight strictly less than $\left|V_0\right|+4$, which will lead to a contradiction.
Similar to how we proceeded in Case 2 above, we first let $g'(u)=0$, $g'(v_1)=1$ and  $g'(v_2)=2$; next we assign values of $g'$ to all vertices
from $V_1\cup V_2-\{v_1,v_2\}$; and lastly to all vertices in $V_0-\{u\}$.

Assume first that $v_1,v_2$ belong to different sets $V_1$ and $V_2$, say $v_1\in V_1$ and $v_2\in V_2$. Consider vertices from $V_2-\{v_2\}$. If all of them are adjacent to $v_1$, their value under $g'$ will be $0$. If there exist vertices in $V_2-\{v_2\}$ not adjacent to $v_1$, one of them obtains the value $1$ and all others the value $0$. In a similar manner we assign values of $g'$ to vertices in $V_1-\{v_1\}$.
Now assume that $v_1,v_2$ belong to the same set, say $V_1$. If there exist vertices $a,b\in V_2$ such that $av_1,bv_2\notin V(\overline{G})$, then $g'(a)=1$, $g'(b)=2$, and $g'(w)=0$ for every $w\in V_1\cup V_2-\{v_1,v_2,a,b\}$.
If all vertices in $V_2$ are adjacent to $v_2$ and there exist vertices in $V_2$ not adjacent to $v_1$, then one of them, say $a$, obtains value $1$ and all other vertices in $V_2-\{a\}$ as well as all vertices in $V_1-\{v_1,v_2\}$ obtain value $0$. Similarly, if all vertices in $V_2$ are adjacent to $v_1$ and there exist vertices in $V_2$ not adjacent to $v_2$, then one of them, say $b$, obtains value $2$ and all other vertices in $V_2-\{b\}$ as well as all vertices in $V_1-\{v_1,v_2\}$ obtain value $0$. In the last case, when all vertices from $V_2$ are adjacent to both $v_1$ and $v_2$, we assign value $0$ to every vertex in $V_1\cup V_2-\{v_1,v_2\}$.

What remains is to define $g'$ on $V_0-\{u\}$.   Let $C=\{z\in V_1\cup V_2: g'(z)=1\}$ and $D=\{w\in V_1\cup V_2: g'(w)=2\}$.  Fix
 a vertex $v$ in  $V_0-\{u\}$. If $v$ has a neighbor in $C$ as well as in $D$, then $g'(v)=0$. If $v$ has a neighbor in $D$ but not in $C$, then $g'(v)=1$.  If $v$ has a neighbor in $C$ but not in $D$, then $g'(v)=2$. For each $w \in V_0-\{u,v\}$ we let $g'(w)=g(w)$.

As in Case 2 above it is easy to see that the function $g'$ is a $2$RiDF of $\overline{G}$ and $w(g')< \left|V_0\right|+4$.
This contradiction proves the claim that if $u \in V_0$, then $u$ has at most one neighbor in $V_1\cup V_2$.

Suppose there exist at least two vertices in $V_0$, say $u$ and $v$, such that $uz,vw\in E(\overline{G})$  for some  $z,w\in V_1\cup V_2$.  Observe
that $N_G(u)=V(G)-\{u,z\}$ and $N_G(v)=V(G)-\{v,w\}$. Assume first that $z\neq w$. One can easily verify that $f':V(G)\rightarrow \{0,1,2\}$ defined by $f'(u)=f'(z)=1,f'(v)=f'(w)=2$, and $f'(a)=0$ otherwise, is a $2$RiDF on $G$. This implies $\ri(G)\leq 4$, hence $\ri(\overline{G})\geq n$ which in fact gives  $\ri(\overline{G})= n$. By Lemma~\ref{un},  $\ri(G)= 2$, which leads to a contradiction. On the other hand, if $z=w$, then $f':V(G)\rightarrow \{0,1,2\}$ defined by $f'(u)=f'(z)=1,f'(v)=2$, and $f'(a)=0$ otherwise, is a $2$RiDF of $G$.  Hence $\ri(G)\leq 3$, and $\ri(G)+\ri(\overline{G})\leq n+3$, which is again a contradiction.

Consequently, we are led   to the situation in the graph $\overline{G}$ in which  there is exactly one vertex $u$ in $V_0$ that has (exactly one) neighbor, say $v$  in $V_1$.  However, if $\left|V_0\right|\geq 2$, then $f':V(G)\rightarrow \{0,1,2\}$, $f'(u)=f'(v)=1$, $f'(z)=2$ for some $z\in V_0\setminus\{u\}$, and $f'(a)=0$ otherwise, is a $2$RiDF of $G$.  This is again a contradiction since $\ri(G)\leq w(f')= 3$.
Thus $u$ is the only vertex in $V_0$.

Recall that  $g(u)\neq 0$ and that there are vertices $x,y\in V_1$ and $s,t\in V_2$ such that $xy,st\in E(\overline{G})$ and $g(x)=g(s)=1$ and $g(y)=g(t)=2$. Thus $\ri(\overline{G})=5$. Note also that $xs,yt\in E(G)$. Let $W=\{w \in V_2 : N_G(w) \cap V_1 \not=\emptyset\}$.
Let $f':V(G)\rightarrow \{0,1,2\}$ be defined by  $f'(u)=2$, $f'(w)=0$ for every vertex $w\in W$, and  $f'(v)=1$ for every $v\in V_1\cup (V_2-W)$.
Since $W$ contains at least two vertices, namely $s$ and $t$, it is straightforward to show that $f'$ is a $2$-rainbow independent
dominating function of $G$ such that $w(f') \leq n-2$.  It follows that $\ri(G)\leq n-2$ and as a result $\ri(G)+\ri(\overline{G})\leq n+3$,
which is the final contradiction.

We have shown that the assumption on equality in (\ref{n+4}) leads to a contradiction, thus
\begin{equation}
\label{sharp}
\ri(G)+\ri(\overline{G})\leq n+3.
\end{equation}
Note that if $G$ is a cycle on $5$ vertices, the equality in (\ref{sharp}) is attained. Any graph of order 3 attains the lower bound.
\end{proof}

The only graphs that attain the lower bound in Theorem~\ref{thm:sum} have order 3.  Indeed, suppose that $G$ has order $n$ where
$n>3$ and that $\ri(G)+\ri(\overline{G})=5$.  We may assume without loss of generality that $\ri(G)=2$.
Let $(\{x_3,\ldots,x_n\},\{x_1\},\{x_2\})$ be a $2$RiDF-partition of $G$.  The vertices $x_1$ and $x_2$ are either isolated
or induce a component of order 2 in $\overline{G}$,
and it follows that $\ri(\overline{G})=2+\ri(\overline{G}\langle\{x_3,\ldots,x_n\}\rangle) \ge 4$ since $n \ge 4$.

\section{Concluding remarks}

In this paper we have presented a new domination concept and explored some of its basic properties. A number of natural  questions remain unanswered. One of these is whether the $5$-cycle is the only graph for which the upper bound is attained in the Nordhaus-Gaddum type inequality. Further, we observed (in Proposition \ref{prop:bounds}) that $i(G) \le \rik(G)$ for any graph $G$ and positive integer $k$, and presented some families of graphs for which the equality holds. In addition, we found a property that must hold if $i(G) = \rik(G)$ (see Corollary \ref{independent}). A characterization of graphs for which the latter equality holds remains open. It would also be interesting to explore algorithmic aspects of computing the $k$-rainbow independent domination number. It is quite likely that the problem of deciding if a graph has a $k$-rainbow independent dominating function of a given weight is NP-complete. However, it would be interesting to consider this question for specific families of graphs as well.

\section{Acknowledgements}

The first author was partially supported
by Slovenian research agency ARRS, program no. P1-0297,
project no. J1-7110. The
second author acknowledges support by a grant from the Simons Foundation
(\# 209654 to Douglas F. Rall). The last author was partially supported by Slovenian research agency ARRS, program no. P1-00383,
project no. L1-4292.

\end{document}